\documentclass[12pt, twoside, leqno]{article}

\usepackage{amsmath,amsthm}
\usepackage{amssymb}

\newtheorem{theorem}{Theorem}[section]
\newtheorem{corollary}[theorem]{Corollary}

\newtheorem{claim}[theorem]{Claim}

\theoremstyle{definition}
\newtheorem{definition}[theorem]{Definition}
\newtheorem{remark}[theorem]{Remark}

\newtheorem*{xthrm}{Theorem}

\numberwithin{equation}{section}

\begin{document}

%%%%% To ease editing, for IMPAN journals add:

\baselineskip=17pt

%%%%%%%%%%%%%%%%

\title{On the Pointwise Lyapunov Exponent of Holomorphic Maps}

\author{Israel Or Weinstein}

\date{}

\maketitle

%% Classification and key words; note that the 2010 classification is used:

\renewcommand{\thefootnote}{}



\renewcommand{\thefootnote}{\arabic{footnote}}
\setcounter{footnote}{0}

%%%%%%%%

\begin{abstract}
We prove that for any holomorphic map, and any bounded orbit which
does not accumulate to a singular set nor to an attracting cycle,
its lower Lyapunov exponent is non-negative. The same result holds
for unbounded orbits too, for maps with a bounded singular set. Furthermore,
the orbit may accumulate to infinity or a singular set, as long as
it is slow enough.
\end{abstract}

\section{Introduction}
An important characteristic of a chaotic system is its
sensitivity to initial conditions. A quantitative measure of this phenomenon
is a positive Lyapunov exponent of an orbit in a dynamical system. In this paper we study the Lyapunov exponent of holomorphic
dynamical systems: Let $f:V\to V'$ be a holomorphic map between open
sets $V\subseteq V'\subseteq\mathbb{C}$. For every initial point
$z_{0}\in V$, as long as it is well-defined, we call $\{z_{0},f(z_{0}),f^{2}(z_{0})...\}$
the orbit of $z_{0}$ under $f$ and denote $z_{n}=f^{n}(z_{0})$.
The\textbf{ lower Lyapunov exponent }of $f$ at the point $z_{0}$
is defined by

\[
\underline{\chi}_{f}(z_{0})=\liminf_{n\to\infty}\frac{1}{n}\sum_{i=0}^{n-1}\log|f'(z_{i})|
\]
Any point that belongs to the basin of an attracting cycle has a negative
Lyapunov exponent.

It is known that the existence of singular values has a significant
influence on the complexity of the dynamical system. The most simple,
and most broadly studied, holomorphic dynamical systems are those
with one singular value: unicritical polynomial maps $z\to z^{d}+c$
and exponential maps $z\rightarrow ae^{z}+c$. For these families
of maps G. Levin, F. Przytycki and W. Shen \cite{levin2014lyapunov}
proved that $\underline{\chi}_{f}(c)\ge0$. More generally, they proved:
\begin{xthrm}[{\cite[Theorem 1.3]{levin2014lyapunov}}]
 Let $f:V\to V'$ be a holomorphic map between open sets $V\subseteq V'\subseteq\mathbb{C}$.
Assume there is a unique point $c\in V'$ such that $f:V\backslash f^{-1}(c)\to V'\backslash\{c\}$
is an unbranched covering map. Assume the orbit $z_{0}=c,\ z_{1}=f(z_{0}),\dots$
is well-defined and $B(z_{i},\delta)\subseteq V'$ for every $i\ge0$
and some $\delta>0$. If $c$ does not belong to the basin of an attracting
cycle then $\underline{\chi}_{f}(c)\ge0$.
\end{xthrm}
To prove this they used a telescopic-like construction (detailed in
Section \ref{sec:Proof-of-Theorem} of the present paper) to estimate
the derivatives $\left(f^{n}\right)^{\prime}(z_{0})$ as a function
of $n$ and $\delta$. In this paper we further develop the argument
of \cite[Lemma 2.2]{levin2014lyapunov} to get a better estimate for
the lower bound of these derivatives. Our estimate holds even for
a map with an arbitrary singular set, under the assumption that the
orbit does not accumulate to this set:
\begin{theorem}  
\label{thm:MainResult}Let $f:V\to V'$ be a holomorphic map between
open sets $V\subseteq V'\subseteq\mathbb{C}$. Let $\mathcal{S}\subset V^{\prime}$
be a relatively closed set such that $f:V\backslash f^{-1}(\mathcal{S})\to V'\backslash\mathcal{S}$
is an unbranched covering map\footnote{One can always take $\text{S=sing}(f^{-1})=\overline{\{s|s\text{ is a critical or asymptotic value of }f\}}$.}.
Let $z_{0}\in V$ be a point with a well-defined orbit $z_{0},z_{1}=f(z_{0}),\dots$
such that $B(z_{i},\delta)\subseteq V'\backslash\mathcal{S}$ for
every $i\ge0$ and some $\delta>0$. Assuming that either $\left\{ z_{i}\right\} _{i=0}^{\infty}$
or $\mathcal{S}$ is bounded, if $z_{0}$ does not belong to the basin
of an attracting cycle then $\underline{\chi}_{f}(z_{0})\ge0$.
\end{theorem}
\begin{remark}
\label{rem:DeltaGrowth}In \cite[Lemma 2.6]{levin2014lyapunov} it
is also shown that the strict requirement for a fixed $\delta>0$
for the entire orbit can be weakened to a less rigid condition: for
a map $f$ with one singular value ($c$), and $z_{0}\in V$ with
a well-defined orbit that does not belong to the basin of an attracting
cycle, if for any $\alpha>0$ and any large enough $n$, $B(z_{n},e^{-\alpha n})\subseteq V'\backslash\{c\}$
then $\underline{\chi}_{f}(z_{0})\ge0$. We derive yet a weaker condition:
for a map $f$ with a bounded singular set ($\mathcal{S}$), if there
are $\kappa>0,\ \beta<\frac{1}{2}$ such that $B(z_{n},\kappa n^{-\beta})\subseteq V'\backslash\mathcal{S}$
for every $n$ then $\underline{\chi}_{f}(z_{0})\ge0$. In the general
case of an arbitrary singular set, the growth of $|z_{n}|$ should
also be taken into account: it is needed that $\min_{i,j\le n}\frac{\min_{s\in\mathcal{S}\cup\mathbb{C}\backslash V'}|z_{i}-s|}{|z_{j}|}\ge\kappa n^{-\beta}$
for every $n$.
\end{remark}
Proofs for Theorem \ref{thm:MainResult} and Remark \ref{rem:DeltaGrowth}
are provided in Section \ref{sec:BoundedTypeAndProof}. Before that,
in the next section, we develop the following, more precise, estimate,
which implies Theorem \ref{thm:MainResult} for bounded orbits:
\begin{theorem}
\label{thm:DerivativeEstimation}Let\textbf{ $V$, $V^{\prime}$,
$\mathcal{S}$} and $f$ be as in Theorem \ref{thm:MainResult}. Let
$z_{0}\in V$ and $n\in\mathbb{N}$ such that the orbit $z_{i}= f^{i}(z_{0})$
is well-defined for any $0\le i\le n$ and $z_{0}$ does not belong
to the basin of an attracting cycle. Define
\footnote{Here and below $d(\circ,\circ)$ means the (minimal) Euclidean distance in the plane.}
$\delta_{n}=\min\left\{ \frac{1}{2},\ \min_{0\le i\le n}d(z_{i},\mathcal{S}),\ \min_{0\le i\le n}d(z_{i},\mathbb{C}\backslash V')\right\} $,
$D_{n}=\max_{0\le i\le n}|z_{i}|+1$. Assume $\delta_{n}>0$,
then for every $1>\gamma>0$:
\[
\left|(f^{n})^{\prime}(z_{0})\right|\ge\rho_{n}^{-1}\exp\left[-C\gamma^{-2}\rho_{n}^{2+\gamma}n^{\frac{4+\gamma}{5}}\right]
\]
 where $C$ is an absolute constant, $\rho_{n}=4\frac{D_{n}+M_{f}}{\delta_{n}}$
and $M_{f}$ is a constant that depends only on $f$: $M_{f}=\inf_{p\in\mathbb{C}\setminus V}|p|+1$
if $V\ne\mathbb{C}$ and $M_{f}=\inf_{P\in\Pi}\max_{p\in P}|p|+1$ if $V=\mathbb{C}$
where $\Pi$ is the set of cycles of $f$.
\end{theorem}

\section{\label{sec:Proof-of-Theorem}Proof of Theorem \ref{thm:DerivativeEstimation}}

Let $f$, $z_{i}$, $\delta_{n}$, $D_{n}$ etc. be as in Theorem
\ref{thm:DerivativeEstimation}. Note that $\rho_{n}=4\frac{D_{n}+M_{f}}{\delta_{n}}$
with $D_{n}\ge1$, $M_{f}\ge1$ and $\delta_{n}\le\frac{1}{2}$ so
$\rho_{n}\ge16$.

We use the construction presented in \cite{levin2014lyapunov}:
\begin{definition}
\label{def:U}For every $0\le i\le n$ define $0<\tau_{i}\le\delta_{n}$
to be the maximal possible radius such that there exists a neighborhood
$U_{i}$ of $z_{i}$ where $f^{n-i}:U_{i}\to B(z_{n},\tau_{i})$ is
a conformal isomorphism.

One can easily prove that these neighborhoods are well defined, that
$(\tau_{i})_{i=0}^{n}$ is non-decreasing, and that for every $0\le i<n$
with $\tau_{i}<\tau_{i+1}$ there exists am $s_{i}\in\mathcal{S}$
such that $s_{i}\in\partial f(U_{i})$.
\end{definition}
The map $f^{n}:U_{0}\to B(z_{n},\tau_{0})$ is univalent and thus,
by Koebe Quarter Theorem:
\begin{corollary}
\label{cor:KoebeBound}
\begin{equation}
|(f^{n})'(z_{0})|\ge\frac{\tau_{0}}{4d(z_{0},\partial U_{0})}
\end{equation}
\end{corollary}
Let us bound the denominator. As in \cite[Eq. 2.7]{levin2014lyapunov}
we use:
\begin{claim}
\label{claim:M}
\begin{equation}
U_{i}\not\supset\overline{B(0,M_{f})},\quad\forall0\le i\le n
\end{equation}
\end{claim}
\begin{proof}
In any case $M_{f}\ge1>\frac{1}{2}\ge\delta_{n}$ so for $i=n$ the
disk $U_{n}=B(z_{n},\delta_{n})$ can not contain the larger disk
$\overline{B(0,M_{f})}$. Now assume $i<n$, for $V\ne\mathbb{C}$ - $M_{f}=\inf_{p\in\mathbb{C}\setminus V}|p|+1$,
so $\overline{B(0,M_{f})}$ contains a point $p\notin V\supseteq U_{i}$.
For entire map we have defined $M_{f}=\inf_{P\in\Pi}\max_{p\in P}|p|+1$
where $\Pi$ is the set of cycles of $f$. By theorem of Fatou there
are such cycles for every entire map which is not of the form $f(z)=z+c$
(where $|(f^{n})^{\prime}|=1$ so the theorem holds for maps of this
form). Let $\frac{1}{4}>\epsilon>0$ so there is a cycle $P\in\Pi$
with $\max_{p\in P}|p|+1<M_{f}+\epsilon$, if the claim does not hold
then $U_{i}\supset\overline{B(0,M_{f})}\supsetneq P$ so
\[
P=f^{n-i}(P)\subset f^{n-i}(U_{i})=B(z_{n},\tau_{i})\subseteq B(z_{n},\delta_{n})
\]

i.e. $|z_{n}-p|<\delta_{n}$ for every $p\in P$ and therefore $|z_{n}|<\delta_{n}+(M_{f}-1+\epsilon)$.
$\delta_{n}\le\frac{1}{2}$ so $\overline{B(z_{n},\delta_{n})}\subseteq\overline{B(0,M_{f}+\epsilon)}$.
It holds for evey small $\epsilon$ so also $\overline{B(z_{n},\delta_{n})}\subseteq\overline{B(0,M_{f})}$
which we assumed to be contained in $U_{i}.$

\[
f^{n-i}(U_{i})=B(z_{n},\tau_{i})\subsetneqq\overline{B(z_{n},\delta_{n})}\subseteq\overline{B(0,M_{f})}\subset U_{i}
\]

which implies by the Schwarz lemma that $z_{i}$ , hence $z_{0}$,
is contained in the basin of an attracting cycle of $f$, a contradiction.
\end{proof}
Claim \ref{claim:M} implies that for every $0\le i\le n$ there is
a point $p$ such that $|p|\le M_{f}$ and $p\not\in U_{i}$. Therefore
$d(z_{i},\partial U_{i})\le M_{f}+|z_{i}|$. In particular $d(z_{0},\partial U_{0})\le M_{f}+|z_{0}|$.

To get the desired bound on $|\left(f^{n}\right)^{\prime}|$ we need
to give a lower bound for $\tau_{0}$. Let $m_{i}=\log\frac{\tau_{i+1}}{\tau_{i}}$
for every $0\le i<n$,
\begin{equation}
\log\tau_{0}=\log\left[\tau_{n}\cdot\frac{\tau_{n-1}}{\tau_{n}}\cdots\frac{\tau_{0}}{\tau_{1}}\right]=\log\delta_{n}-\sum_{i=0}^{n-1}m_{i}
\end{equation}
To estimate this value define for every $m\ge0$: $\mathcal{I}_m=\left\{ 0\le i<n:\ m_{i}\ge m\right\} $
and consider the tail distribution function $\ensuremath{F_{z_{0},n}:[0,\infty)\to\{0,1,...,n\}}$:
\begin{equation}
F_{z_{0},n}(m)=\#\left\{ 0\le i<n:\ m_{i}\ge m\right\} =\#\mathcal{I}_m
\end{equation}

Define the sequence $\mbox{(\ensuremath{m^{i})_{0}^{k+1}}}$ to be
the (unique) elements of $\{0\}\cup\left\{ m_{i}\right\} _{i=0}^{n-1}\cup\left\{ \max\{m_{i}\}+1\right\} $
arranged in a (strictly) monotonically order, i.e $m^{0}=0$, $m^{1}=\min\{m_{i}\}$,
$\dots$, $m^{k}=\max\{m_{i}\}$, $m^{k+1}=\max\{m_{i}\}+1$. With
this definition, for every $i\le k$: $\#\{j:\ m_{j}=m^{i}\}=\ensuremath{F_{z_{0},n}(m^{i})-F_{z_{0},n}(m^{i+1})}$.
Thus, using summation by parts (Abel transformation), we get:

\[
\sum_{i=0}^{n-1}m_{i}=\sum_{i=0}^{k}m^{i}(F_{z_{0},n}(m^{i})-F_{z_{0},n}(m^{i+1}))=
\]
\[
=m^{0}F_{z_{0},n}(m^{0})-m^{k+1}F_{z_{0},n}(m^{k+1})+\sum_{i=1}^{k}F_{z_{0},n}(m^{i})(m^{i}-m^{i-1})=
\]
\[
=0\cdot F_{z_{0},n}(0)-m^{k+1}\cdot0+\int_{0}^{\infty}F_{z_{0},n}(m)dm
\]
So the bound for $|(f^{n})^{\prime}|$ turns into:
\begin{equation}
\begin{split}
|(f^{n})^{\prime}(z_{0})|\ge\frac{1}{4}\frac{\delta_{n}}{M_{f}+|z_{0}|}\exp\left[-\int_{0}^{\infty}F_{z_{0},n}(m)dm\right]\\
\ge\rho_{n}^{-1}\exp\left[-\int_{0}^{\infty}F_{z_{0},n}(m)dm\right]\label{eq:BaseDerBound}
\end{split}
\end{equation}

Note $F_{z_{0},n}$ gives an upper bound on the number of possible
values of $i$ such that $m\le m_{i}=\log\frac{\tau_{i+1}}{\tau_{i}}$.
The right hand side of this inequality is the modulus\footnote{It is more standard to define modulus as $\frac{1}{2\pi}\log\frac{\tau_{i+1}}{\tau_{i}}$.}
of the annulus $\left\{ z:\ \tau_{i+1}<|z-z_{n}|<\tau_{i}\right\} $.
Directly from the definition of $U_{i}$:
\[
\begin{array}{ll}
f^{n-(i+1)}(U_{i+1}) & =B(z_{n},\tau_{i+1})\\
f^{n-(i+1)}(f(U_{i}))=f^{n-i}(U_{i}) & =B(z_{n},\tau_{i})
\end{array}
\]
and the maps $f^{n-i-1}$ are conformal so preserve modulus: $\bmod(U_{i+1}\backslash f(U_{i}))=\log\frac{\tau_{i+1}}{\tau_{i}}=m_{i}$.
We will use this fact extensively in the following known theorem (the proof is included for the reader’s convenience):
\begin{theorem}[cf. \cite{ahlfors2010conformal}]
\label{thm:Teichmuller}If $\mathcal{A}\subset\mathbb{C}$ is a doubly
connected region with finite modulus $m$ that separates the pair
$\left\{ e_{1},\ e_{2}\right\} $ from the pair $\left\{ e_{3},\ \infty\right\} $
then
\[
\left|e_{3}-e_{1}\right|\ge|e_{2}-e_{1}|\cdot\max\left\{ \frac{1}{16}e^{m}-1,\ 16e^{-\frac{\pi^{2}}{m}}\right\}
\]
\end{theorem}
\begin{proof}
By Teichmuller Extremal Modulus Theorem (e.g. \cite[Theorem 4-7]{ahlfors2010conformal}),
of all doubly connected regions that separate the pair $\left\{ 0,-1\right\} $
from a pair $\left\{ w_{0},\ \infty\right\} $ with $|w_{0}|=R$,
the one with the greatest modulus is the complement of the segments
$[-1,0]$ and $\left[R,\ +\infty\right]$. Denote the modulus of this
region by $\Lambda(R)$. This function is known to be bounded by:
\[
R-1\le\frac{e^{\Lambda(R)}}{16}-1\le R
\]
and $\Lambda(R)\Lambda(R^{-1})=\pi^{2}$ so also:
\[
R\ge16e^{-\frac{\pi^{2}}{\Lambda(R)}}
\]

With the map $z\to\frac{z-e_{1}}{e_{1}-e_{2}}$ the region $\mathcal{A}$
is mapped to a region that separates $\left\{ 0,-1\right\} $ from
$\left\{ \frac{e_{3}-e_{1}}{e_{1}-e_{2}},\ \infty\right\} $, and
therefore $m\le\Lambda\left(\left|\frac{e_{3}-e_{1}}{e_{1}-e_{2}}\right|\right)$,
i.e.
\[
\left|\frac{e_{3}-e_{1}}{e_{1}-e_{2}}\right|\ge\max\left\{ \frac{1}{16}e^{m}-1,\ 16e^{-\frac{\pi^{2}}{m}}\right\}
\]
\end{proof}
First application of this theorem is to change the interval of integration
in Eq. \ref{eq:BaseDerBound} to be finite.
\begin{corollary}
Set $m_{\max}=2+\log\rho_{n}$, then:
\begin{equation}
\ensuremath{\forall m\ge m_{\max}:\quad F_{z_{0},n}(m)\equiv0}
\end{equation}
\end{corollary}
\begin{proof}
Assume $F_{z_{0},n}(m)\neq0$ - this means that there exists an $i<n$
with $m_{i}\ge m_{\max}$. As we have seen, by the definition of modulus:
\[
\bmod(U_{i+1}\backslash f(U_{i}))=\log\left(\frac{\tau_{i+1}}{\tau_{i}}\right)=m_{i}\ge m_{\max}
\]
$z_{i+1}\in f(U_{i})$, $m_{i}>0$ so there is a singular $s_{i}\in\partial f(U_{i})$
(so $|s_{i}-z_{i+1}|\ge\delta_{n}$ by definition) and by Claim \ref{claim:M}
there is a point $p\notin U_{i+1}$ with $|p|\le M_{f}$. So by Theorem
\ref{thm:Teichmuller}:
\[
\left|p-z_{i+1}\right|\ge|s_{i}-z_{i+1}|\cdot\left(\frac{1}{16}e^{m_{i}}-1\right)\ge\delta_{n}\cdot\left(\frac{e^{2}}{16}\rho_{n}-1\right)
\]

but $\left|p-z_{i+1}\right|\le|p|+|z_{i+1}|\le M_{f}+D_{n}$, $\rho_{n}=4\frac{M_{f}+D_{n}}{\delta_{n}}$
and $\rho_{n}\ge16$ so we get the contradiction $M_{f}+D_{n}\ge\delta_{n}\cdot\rho_{n}\cdot\left(\frac{e^{2}-1}{16}\right)=(M_{f}+D_{n})\left(\frac{e^{2}-1}{4}\right)$.
\end{proof}
Thus we can bound our integral by $\int_{0}^{\infty}F_{z_{0},n}(m)dm=\int_{0}^{m_{\max}}F_{z_{0},n}(m)dm$.
Now let us start to construct a bound for $F_{z_{0},n}$ by showing
that we can get an explicit lower bound for distance between elements
of $\mathcal{I}_m$.
\begin{claim}
\label{lem:Alpha}For every $m>0$ and $i\in\mathcal{I}_m$:
\begin{equation}
B\left(z_{i+1},\frac{\delta_{n}}{\alpha(m)}\right)\subset f(U_{i})
\end{equation}
with $\alpha(m)=\left(\frac{2}{m}+1\right)^{2}$.
\end{claim}
\begin{proof}
Let $i\in\mathcal{I}_m$ so $m_{i}\ge m>0$. By Definition \ref{def:U} an inverse
branch $\ensuremath{g_{0}:B(z_{n},}\tau_{i+1})\to U_{i+1}$ of $f^{n-i-1}$
is a well defined conformal isomorphism. Let $\ensuremath{g(w)=g_{0}(\tau_{i+1}w+z_{n})}$
so that $g:\mathbb{D}\to U_{i+1}$. In particular:
\[
\ensuremath{g\left(B\left(0,\frac{\tau_{i}}{\tau_{i+1}}\right)\right)=g_{0}(B(z_{n},\tau_{i}))=f(U_{i})}.
\]
By the Koebe Distortion Theorem, for any $w\in\mathbb{D}$ ($g'(0)=\frac{\tau_{i+1}}{(f^{n-i-1})'(z_{i+1})}\ne0$),
\begin{equation}
\frac{|w|}{(1-|w|)^{2}}\ge\left|\frac{g(w)-g(0)}{g'(0)}\right|\ge\frac{|w|}{(1+|w|)^{2}}\label{eq:distortion}
\end{equation}
The circle $\ensuremath{|w|=\frac{\tau_{i}}{\tau_{i+1}}=e^{-m_{i}}}$
is mapped by $g$ onto:
\[
g_{0}(\partial B(z_{n},\tau_{i}))=\partial f(U_{i})
\]
and since $m_{i}\ge m>0$ there is $s_{i}\in\mathcal{S}$ such that
$s_{i}\in\partial f(U_{i})$. But by the definition of $\delta_{n}$,
$|z_{i+1}-s_{i}|\ge\delta_{n}$ so the left hand side of (\ref{eq:distortion})
yields:
\[
\frac{e^{-m_{i}}}{(1-e^{-m_{i}})^{2}}\ge\max_{|w|=e^{-m_{i}}}\left|\frac{g(w)-g(0)}{g'(0)}\right|=\left|g'(0)\right|^{-1}\max_{v\in\partial f(U_{i})}\left|v-z_{i+1}\right|\ge\frac{\delta_{n}}{\left|g'(0)\right|}
\]
This inequality and the right hand side of (\ref{eq:distortion})
yields (after multiplication by $\left|g'(0)\right|$):
\begin{equation*}
\begin{split}
\min_{v\in\partial f(U_{i})}\left|v-z_{i+1}\right|
\ge\left|g'(0)\right|\frac{e^{-m_{i}}}{(1+e^{-m_{i}})^{2}}
&\ge\delta_{n}\frac{(1-e^{-m_{i}})^{2}}{e^{-m_{i}}}\frac{e^{-m_{i}}}{(1+e^{-m_{i}})^{2}}\\
&=\delta_{n}\left(\frac{1-e^{-m_{i}}}{1+e^{-m_{i}}}\right)^{2}
\end{split}
\end{equation*}

For $x\ge0$, $e^{x}\ge1+x$ hence $\frac{1-e^{-x}}{1+e^{-x}}\ge\frac{1-\frac{1}{1+x}}{1+\frac{1}{1+x}}=\left(\frac{2}{x}+1\right)^{-1}$
so we get:

\[
d(z_{i+1},\partial f(U_{i}))\ge\delta_{n}\left(\frac{2}{m_{i}}+1\right)^{-2}=\frac{\delta_{n}}{\alpha(m_{i})}\ge\frac{\delta_{n}}{\alpha(m)}
\]
\end{proof}
Denote $\ensuremath{\mathcal{I}_m=\{i_{1}<i_{2}<...<i_{F_{z_{0},n}(m)}\}}$.
\begin{claim}
\label{claim:E}For $0<j<k\le F_{z_{0},n}(m)$ with $k-j\ge E(m)$:
\begin{equation}
|z_{i_{j}+1}-z_{i_{k}+1}|\ge\frac{\delta_{n}}{2\alpha(m)}
\end{equation}
where $\ensuremath{E(m)=\left\lfloor m^{-1}\log\left[9\rho_{n}\alpha(m)\right]\right\rfloor }$.
\end{claim}
\begin{proof}
The conformal map $f^{n-i_{k}-1}$ maps the annulus: $\ensuremath{\mathcal{A}=U_{i_{k}+1}\backslash f^{i_{k}-i_{j}+1}(U_{i_{j}})}$
onto the geometric ring with radii $\ensuremath{\tau_{i_{k}+1},\ \tau_{i_{j}}}$
around $z_{n}$. According to our choice of $\ensuremath{i_{j},i_{k}}$,
we have obtained a minimum ratio between these radii (i.e. modulus):
\begin{equation}
\bmod\mathcal{A}=\log\frac{\tau_{i_{k}+1}}{\tau_{i_{j}}}\ge m\cdot(k+1-j)\ge m(E(m)+1)\ge\log\left[9\rho_{n}\alpha(m)\right]\label{eq:ErR}
\end{equation}

$z_{i_{k}+1}\in f^{i_{k}-i_{j}+1}(U_{i_{j}})$ and by Claim \ref{claim:M}
there is a point $p\notin U_{i_{k}+1}$ with $|p|\le M_{f}$ so by
Theorem \ref{thm:Teichmuller}, for every $w\in\overline{f^{i_{k}-i_{j}+1}(U_{i_{j}})}$:
\[
|w-z_{i_{k}+1}|\le\left|p-z_{i_{k}+1}\right|\cdot\left(\frac{1}{16}e^{m_{i}}-1\right)^{-1}\le(M_{f}+D_{n})\cdot\left(\frac{9\rho_{n}\alpha(m)}{16}-1\right)^{-1}
\]
and $\rho_{n}\ge16$, $\alpha(m)\ge1$ therefore
\[
|w-z_{i_{k}+1}|\le\frac{16}{9-1}\cdot\frac{M_{f}+D_{n}}{\alpha(m)\rho_{n}}=\frac{\delta_{n}}{2\alpha(m)}
\]
so $\ensuremath{\overline{f^{i_{k}-i_{j}+1}(U_{i_{j}})}\subseteq\overline{B\left(z_{i_{k}+1},\frac{\delta_{n}}{2\alpha(m)}\right)}}$.

Assume now that Claim \ref{claim:E} dose not hold:
\[
\ensuremath{|z_{i_{j}+1}-z_{i_{k}+1}|<\frac{\delta_{n}}{2\alpha(m)}}
\]
then, by Claim \ref{lem:Alpha}:
\[
\overline{f^{i_{k}-i_{j}+1}(U_{i_{j}})}\subset B\left(z_{i_{j}+1},\ \frac{\delta_{n}}{\alpha(m)}\right)\subset f(U_{i_{j}})
\]
After normalization and use of Schwarz lemma we get there exists an
attracting fix point of $\ensuremath{f^{i_{k}-i_{j}+1}}$ which all
$\ensuremath{B(z_{i_{j}+1},\frac{\delta_{n}}{\alpha})}$ attracted
toward this point. Hence $z_{0}$ is in the basin of attraction of
some attractive periodic orbit of $f$, with a contradiction to the
assumptions on $z_{0}$.
\end{proof}
We chose $E(m)$ so that the distance between the elements of the
set
\[
\ensuremath{\mathcal{K}_{m}=\left\{ z_{i_{(k\cdot E)}+1}:\ 0\le k<\frac{\#\mathcal{I}_m}{E(m)}\right\} }
\]
would be at least $\frac{\delta_{n}}{2\alpha}$ from each other. On
the other hand, these elements are also bounded in $\overline{B(0,D_{n})}$.
These two characteristics of $\mathcal{K}_{m}$ help to bound the
number of elements in it:
\begin{claim}
\label{claim:F} For any $m>0$
\begin{equation}
F_{z_{0},n}(m)\le E(m)\left(\rho_{n}\alpha(m)\right)^{2}
\end{equation}
\end{claim}
\begin{proof}
For $m\ge m_{\max}$ we know $F_{z_{0},n}(m)=0$, so the inequality
holds. $\alpha(m)=\left(1+\frac{2}{m}\right)^{2}>1$ for every $m$
so for $m<m_{\max}=\log\left(e^{2}\rho_{n}\right)<\log\left(9\rho_{n}\alpha(m)\right)$
we have
\begin{align*}
E(m)
& =\left\lfloor m^{-1}\log\left(9\rho_{n}\alpha(m)\right)\right\rfloor
\ge\left\lfloor m_{\max}^{-1}\cdot\log\left(9\rho_{n}\right)+m_{\max}^{-1}\cdot\log1\right\rfloor
\ge1
\end{align*}

Then assume $m<m_{\max}$ so $\frac{\#\mathcal{I}_m}{E(m)}=\frac{F_{z_{0},n}(m)}{E(m)}$
is well defined. Now - geometrically: draw discs with radii $\frac{\delta_{n}}{4\alpha}$
around the points of $\mathcal{K}_{m}$. These discs can not intersect,
because otherwise let $z_{i},\ z_{j}\in\mathcal{K}_{m}$ be the centers
of two such discs with a common point $q$ then

\[
\ensuremath{|z_{i}-z_{j}|\le|z_{i}-q|+|z_{j}-q|<\frac{\delta_{n}}{4\alpha}+\frac{\delta_{n}}{4\alpha}=\frac{\delta_{n}}{2\alpha}}
\]
which is a contradiction to Claim \ref{claim:E}.

So the discs have a total area of $\#\mathcal{K}_{m}\cdot\pi\left(\frac{\delta_{n}}{4\alpha}\right)^{2}=\left\lfloor \frac{F_{z_{0},n}}{E}\right\rfloor \pi\left(\frac{\delta_{n}}{4\alpha}\right)^{2}$.
The centers of these discs are points in the orbit of $z_{0}$ so
by the definition of $D_{n}$ all the discs must be contained inside
the disc $B\left(0,D_{n}+\frac{\delta_{n}}{4\alpha}\right)$ which
has an area of $\pi\left(D_{n}+\frac{\delta_{n}}{4\alpha}\right)^{2}$.
So

\[
\pi\left(D_{n}+\frac{\delta_{n}}{4\alpha}\right)^{2}\ge\left(\frac{F_{z_{0},n}(m)}{E(m)}-1\right)\pi\left(\frac{\delta_{n}}{4\alpha(m)}\right)^{2}
\]

\[
F_{z_{0},n}(m)\le E(m)\left(4\alpha(m)\frac{D_{n}}{\delta_{n}}+1\right)^{2}+E(m)\le E(m)\left(\rho_{n}\alpha(m)\right)^{2}
\]
where the last step is because $\alpha(m)\ge1$ for every $m$, $\delta_{n}\le\frac{1}{2}$
and $M_{f}\ge1$ so $4\frac{M_{f}}{\delta_{n}}\alpha(m)>2$.
\end{proof}
So we got an expression (call it $F(m)$) which depends only on $m$
and bounds $F_{z_{0},n}(m)$. Explicitly:
\[
 F(m) = m^{-1}\log\left[9\rho_{n}\alpha(m)\right]\Bigg(\rho_{n}\alpha(m)\Bigg)^{2} \ge  F_{z_{0},n}(m)
\]
Recall we need to bound $\int_{0}^{\infty}F_{z_{0},n}(m)dm$. Divide
this integration interval into four parts:
\begin{enumerate}
\item $[m_{\max},\infty)$: where $F_{z,n}\equiv0$ so $\int_{m_{\max}}^{\infty}F_{z_{0},n}(m)dm=0$.
\item $[2,m_{\max})$: here $\alpha(m)\le\left(\frac{2}{2}+1\right)^{2}=4$
so we can remove the dependence of $F$ on $m$ and get ($\rho_{n}\ge16)$:
\[
F(m)\le2^{-1}\left(\log\left[36\rho_{n}\right]\right)\Bigg(4\rho_{n}\Bigg)^{2}\le30\rho_{n}^{2}\log\rho_{n}
\]
\[
\int_{2}^{m_{\max}}F_{z_{0},n}(m)dm\le(m_{\max}-2)\cdot30\rho_{n}^{2}\log\rho_{n}=30(\rho_{n}\log\rho_{n})^{2}
\]
\item $\left(0,a_{n}\right)$ with $a_{n}=\text{\ensuremath{\frac{1}{\sqrt[5]{n}}}}$:
for every $m\ge0$, and in particular in this interval, $F_{z_{0},n}(m)=\#\{0\le i<n:\ m_{i}\ge m\}\le n$.
\[
\int_{0}^{a_{n}}F_{z_{0},n}(m)dm\le\int_{0}^{a_{n}}ndm=n^{\frac{4}{5}}
\]
\item $[a_{n},2)$: in this interval $\alpha(m)\le\left(\frac{2}{m}+\frac{2}{m}\right)^{2}\le16m^{-2}$
so
\[
F(m)=m^{-1}\log\left[9\rho_{n}\cdot16m^{-2}\right]\Bigg(\rho_{n}16m^{-2}\Bigg)^{2}\le4c_{1}\rho_{n}^{2}m^{-5}\log\left[\rho_{n}m^{-1}\right]
\]
for some constant $c_{1}>0$. This expression has an explicit primitive
function:
\[
\int_{a_{n}}^{2}F_{z_{0},n}(m)dm\le\int_{a_{n}}^{2}F(m)dm
\]
\[
\le-\left.\left[c_{1}\rho_{n}^{2}m^{-4}\log(\rho_{n}m^{-1})-\frac{c_{1}}{4}\rho_{n}^{2}m^{-4}\right]\right|_{a_{n}}^{2}
\]
\[
\le\left.c_{1}\rho_{n}^{2}m^{-4}\log\left(\rho_{n}m^{-1}\right)\right|_{m=\text{\ensuremath{\frac{1}{\sqrt[5]{n}}}}}+\frac{c_{1}}{4}\rho_{n}^{2}2^{-4}\le2c_{1}\rho_{n}^{2}n^{\frac{4}{5}}\log(\rho_{n}n^{\frac{1}{5}})
\]
\end{enumerate}
Connecting all the intervals we get our bound:
\[
\int_{0}^{\infty}F_{z_{0},n}(m)dm\le30(\rho_{n}\log\rho_{n})^{2}+n^{\frac{4}{5}}+2c_{1}\rho_{n}^{2}n^{\frac{4}{5}}\log\left(\rho_{n}n^{\frac{1}{5}}\right)
\]

For any given $\gamma>0$ we can use $\log x=\gamma^{-1}\log(x^{\gamma})<\gamma^{-1}x^{\gamma}$
to simplify the last bound as (assume $\gamma<1$):

\begin{equation}
\int_{0}^{\infty}F_{z_{0},n}(m)dm\le C\gamma^{-2}\rho_{n}^{2+\gamma}n^{\frac{4+\gamma}{5}}
\end{equation}

for some constant $C>0$.

Finally:

\begin{equation}
|(f^{n})^{'}(z_{0})|\ge\rho_{n}^{-1}\exp\left[-\int_{0}^{\infty}F_{z_{0},n}(m)dm\right]\ge\rho_{n}^{-1}\exp\left[-C\gamma^{-2}\rho_{n}^{2+\gamma}n^{\frac{4+\gamma}{5}}\right]
\end{equation}

which ends the proof of Theorem \ref{thm:DerivativeEstimation}.

\section{Proof of Theorem \ref{thm:MainResult} and maps of bounded type\label{sec:BoundedTypeAndProof}}

Recall that in Theorem \ref{thm:MainResult} there are a map $f$
and a point $z_{0}$ which fulfill the conditions for Theorem \ref{thm:DerivativeEstimation}
for every $n\in\mathbb{N}$ and that there is some $\delta>0$ such
that $\delta_{n}\ge\delta$ for any $n$. Moreover, either $\sup |z_{n}|<\infty$
or $\sup_{s\in\mathcal{S}}|s|<\infty$.

Let us first handle the case of a bounded orbit. Recall $D_{n}=\max_{0\le i\le n}|z_{i}|+1$, define $D=\sup D_{n}<\infty$. Directly from Theorem
\ref{thm:DerivativeEstimation} with $\gamma=\frac{1}{2}$:

\[
|(f^{n})^{'}(z_{0})|\ge\rho_{n}^{-1}\exp\left[-4C\rho_{n}^{\frac{5}{2}}n^{\frac{9}{10}}\right]
\]

$\rho_{n}=4\frac{D_{n}+M_{f}}{\delta_{n}}\le4\frac{D+M_{f}}{\delta}:=\rho>0$
hence

\begin{equation*}
\begin{split}
\underline{\chi}_{f}(z_{0})
&=\liminf_{n\to\infty}\frac{1}{n}\log|(f^{n})^{'}(z_{0})|\\
&\ge-\left(\log\rho\right)\liminf_{n\to\infty}n^{-1}-4C\rho^{\frac{5}{2}}\liminf_{n\to\infty}n^{-\frac{1}{10}}
=0
\end{split}
\end{equation*}

Before the proof of the second case (bounded $\mathcal{S}$), recall
that in Remark \ref{rem:DeltaGrowth} we have mentioned that even
without the conditions $\delta_{n}\ge\delta>0$ and $D_{n}\ge D>0$
one can get $\underline{\chi}_{f}(z_{0})\ge0$ as long as $\frac{\delta_{n}}{D_{n}}\ge\kappa n^{-\beta}$
for some $\kappa>0$, $\beta<\frac{1}{2}$ for every $n$. We can
try to use Theorem \ref{thm:DerivativeEstimation} as it is to show
that:

Assume whichthe condition $\frac{\delta_{n}}{D_{n}}\ge\kappa n^{-\beta}$
holds for some $\kappa$, $\beta$. $D_{n}\ge1$ so
\[
\rho_{n}=4\frac{M_{f}+D_{n}}{\delta_{n}}\le4(M_{f}+1)\frac{D_{n}}{\delta_{n}}\le4(M_{f}+1)\kappa^{-1}n^{\beta}
\]
which by Theorem \ref{thm:DerivativeEstimation} yields ($c_{2}$,
$c_{3}$ and $c_{4}$ are positive constants that depend on $\beta$,
$\gamma$, $\kappa$ and $f$):
\begin{equation*}
\begin{split}
\frac{1}{n}\log|(f^{n})^{\prime}(z_{0})|
&\ge-c_{2}\left(\frac{\log n}{n}\right)-c_{3}\left(n^{-1}n^{\beta(2+\gamma)}n^{\frac{4+\gamma}{5}}\right)\\
&=-c_{4}\left(n^{-\frac{1}{5}+2\beta+\gamma\left(\beta+\frac{1}{5}\right)}\right)
\end{split}
\end{equation*}
so as long as $\beta<\frac{1}{10}$ one can choose $\gamma<\frac{1-10\beta}{1+5\beta}$
and get a negative power, i.e. $\underline{\chi}_{f}(z_{0})\ge0$.
To allow the faster growth $\kappa n^{-\beta}$ with $\beta<\frac{1}{2}$
we must revisit the end of the proof of Theorem \ref{thm:DerivativeEstimation}:

We have divided the integration interval into four parts: $(0,a_{n})$,
$[a_{n},2)$, $[2,m_{\max})$ and $[m_{\max},\infty)$ with $a_{n}=n^{-\frac{1}{5}}$.
One can leave the two last intervals as they are, but replace $a_{n}$
to be a slower decreasing sequence, for example $a_{n}=\frac{1}{\log n}$.
With this choice $\int_{0}^{a_{n}}F_{z_{0},n}(m)dm\le a_{n}\cdot n=\frac{n}{\log n}$
and (assume $n\ge3$)

\begin{equation*}
\begin{split}
\int_{a_{n}}^{2}F_{z_{0},n}(m)dm
\le\int_{a_{n}}^{2}F(m)dm
&\le\left[c_{1}\rho_{n}^{2}m^{-4}\log\left(\rho_{n}m^{-1}\right)\right]_{m=a_{n}}+\frac{c_{1}}{4}\rho_{n}^{2}2^{-4}\\
&\le2c_{1}\rho_{n}^{2}(\log n)^{4}\log(\rho_{n}\log n)
\end{split}
\end{equation*}

so

\begin{equation*}
\begin{split}
\frac{1}{n}\log|(f^{n})^{\prime}(z_{0})|
&\ge-\frac{\log\rho_{n}}{n}-\frac{1}{n}\int_{0}^{\infty}F_{z_{0},n}(m)dm\\
&\ge-\frac{\log\rho_{n}}{n}-a_{n}-2c_{1}\frac{\rho_{n}^{2}(\log n)^{4}\log(\rho_{n}\log n)}{n}-30\frac{\rho_{n}^{2}\left(\log\rho_{n}\right)^{2}}{n}
\end{split}
\end{equation*}

If $\rho_{n}\le4(M_{f}+1)\kappa^{-1}n^{\beta}$ for some $\beta<\frac{1}{2}$
\begin{equation*}
\begin{split}
&\underline{\chi}_{f}(z_{0})\ge\\
&-\liminf_{n\to\infty}\left[c_{2}\frac{\log n}{n}+\frac{1}{\log n}+c_{5}n^{2\beta-1}\left(\log n\right)^{4}(\log n+\log\log n)+c_{6}n^{2\beta-1}\left(\log n\right)^{2}\right]\\
&=0
\end{split}
\end{equation*}

Until this point we have proved Theorem \ref{thm:MainResult} (and
Remark \ref{rem:DeltaGrowth}) only for the case that the orbit is
bounded or at least it approaches infinity slow enough. Next we will
prove that if the singular set $\mathcal{S}$ is bounded then $|z_{n}|$
is no longer something to bother about. This will also end the proof
of Theorem \ref{thm:MainResult}.

The strength of this extension is that $\mathcal{S}$ is a property
of the map $f$ alone, not the specific orbit $z_{0},\dots$. The
dynamics of entire maps with bounded singular set, known as ''entire
maps of bounded type'', were first investigated by Eremenko and Lyubich
\cite{eremenko1992dynamical}. This class of maps contains all polynomials
and exponents, but also maps with infinity critical values such as
$\frac{\sin z}{z}$. It is also closed under compositions. For maps
with bounded singular set (not only entire maps) we prove the following
variation of Theorem \ref{thm:DerivativeEstimation}:
\begin{theorem}
\label{thm:BoundedType}Let\textbf{ $V$, $V^{\prime}$, $\mathcal{S}$}
and $f$ be as in Theorem \ref{thm:MainResult}. Let $z_{0}\in V$
and $n\in\mathbb{N}$ such that the orbit $z_{i}= f^{i}(z_{0})$
is well-defined for any $0\le i\le n$ and $z_{0}$ does not belong
to the basin of an attracting cycle. Define $\delta_{n}=\min\left\{ \frac{1}{2},\ \min_{0\le i\le n}d(z_{i},\mathcal{S}),\ \min_{0\le i\le n}d(z_{i},\mathbb{C}\backslash V')\right\} $.
If $S_{f}=\sup_{s\in\mathcal{S}}|s|+1<\infty$ and $\delta_{n}>0$
then for $\tilde{\rho_{n}}=4\frac{S_{f}+M_{f}}{\delta_{n}}$:
\[
|(f^{n})'(z_{0})|\ge\frac{1}{4}\frac{\delta_{n}}{M_{f}+|z_{0}|}\exp\left[-C\left(\tilde{\rho_{n}}\log\tilde{\rho_{n}}\right)^{2}\frac{n}{\log n}\right]
\]
where $C$ is an absolute constant and $M_{f}$ is the constant that
depends only on $f$ as defined in Theorem \ref{thm:DerivativeEstimation}.
\end{theorem}
Of course the map and the orbit satisfy the conditions of Theorem
\ref{thm:DerivativeEstimation} so we can use the same claims here.
In particular we use Eq. \ref{eq:BaseDerBound} again (this time we
do not want to use the bound $|z_{0}|\le D_{n}$):
\[
|(f^{n})^{\prime}(z_{0})|\ge\frac{1}{4}\frac{\delta_{n}}{M_{f}+|z_{0}|}\exp\left[-\int_{0}^{\infty}F_{z_{0},n}(m)dm\right]
\]

The main difference is that instead of using a constant bound for
the orbit ($D$), for these maps we can bound the orbit as a function
of $m$:
\begin{claim}
\label{claim:D}For every $m>0$ and every $i\in\mathcal{I}_m$:
\begin{equation}
|z_{i+1}|\le D(m)
\end{equation}

where $D(m)=S_{f}+(M_{f}+S_{f})\cdot\frac{1}{16}e^{\frac{\pi^{2}}{m}}$.
\end{claim}
\begin{proof}
Recall that by the definition of modulus
\[
\bmod(U_{i+1}\backslash f(U_{i}))=\log\frac{\tau_{i+1}}{\tau_{i}}=m_{i}\ge m
\]
also recall that $z_{i+1}\in f(U_{i})$ and there is some $s_{i}\in\partial f(U_{i})\cap\mathcal{S}$
for every $i$ with $m_{i}>0$, and last - we have built $M_{f}$
to have (Claim \ref{claim:M}) $p\notin U_{i+1}$ with $|p|\le M_{f}$,
so Theorem \ref{thm:Teichmuller} yields:
\begin{equation}
|z_{i+1}-s_{i}|\cdot\max\left\{ \frac{1}{16}e^{m}-1,\ 16e^{-\frac{\pi^{2}}{m}}\right\} \le\left|p-s_{i}\right|\le M_{f}+S_{f}\label{eq:ZminusSforD}
\end{equation}
and the claim follows by inverting $16e^{-\frac{\pi^{2}}{m}}$ and
$|z_{i+1}|\le|z_{i+1}-s_{i}|+|s_{i}|\le|z_{i+1}-s_{i}|+S_{f}$.
\end{proof}
In this case we can get ''new'' $m_{\max}$ that does not depend
on $|z_{n}|$:
\begin{corollary}
For every
\begin{equation}
\ensuremath{m>\tilde{m}_{\max}=2+\log\tilde{\rho_{n}}}
\end{equation}
it holds that $\ensuremath{F_{z_{0},n}(m)=0}$.
\end{corollary}
\begin{proof}
Let $m>0$ and $i\in\mathcal{I}_m$, so by Eq. \ref{eq:ZminusSforD} again (this
time with the second bound):
\[
|z_{i+1}-s_{i}|\cdot\left(\frac{1}{16}e^{m}-1\right)\le M_{f}+S_{f}
\]
but $|z_{i+1}-s_{i}|\ge\delta_{n}$ by the definition of $\delta_{n}$
so ($\tilde{\rho_{n}}=4\frac{M_{f}+S_{f}}{\delta_{n}}\ge16$):
\[
e^{m}\le16+16\frac{M_{f}+S_{f}}{|z_{i+1}-s_{i}|}\le5\tilde{\rho_{n}}<e^{2}\tilde{\rho_{n}}
\]
\end{proof}
The function $\alpha(m)$ does not depend on $D_{n}$ so no change
is needed in Claim \ref{lem:Alpha}. In Claim \ref{claim:E} we used
$D_{n}$ to bound $|z_{i_{k}+1}|$ for $i_{k}\in\mathcal{I}_m$, so we can use
$|z_{i_{k}+1}|\le D(m)$ instead. Finally, we have built the bound
$F(m)\ge F_{z_{0},n}(m)$ by using the fact that the elements of the
form $\mathcal{K}_{m}=\left\{ z_{i_{0}+1},z_{i_{E}+1},\dots\right\} $
are bounded in the disk $\overline{B(0,D_{n})}$. Since $i_{0},i_{E},\dots\in\mathcal{I}_{m}$
- this set is bounded in the disk $\overline{B(0,D(m))}$ too. So
all we have to do is to use $|z_{i+1}|\le D(m)$ instead of $|z_{i+1}|\le D_{n}$.
So the function $F$ becomes:
\[
F_{z_{0},n}(m)\le\tilde{F}(m)=m^{-1}\log\left[9\cdot4\frac{D(m)+M_{f}}{\delta_{n}}\alpha(m)\right]\Bigg(4\frac{D(m)+M_{f}}{\delta_{n}}\alpha(m)\Bigg)^{2}
\]
\[
=m^{-1}\log\left[9\tilde{\rho_{n}}\left(1+\frac{1}{16}e^{\frac{\pi^{2}}{m}}\right)\alpha(m)\right]\tilde{\rho_{n}}^{2}\left(1+\frac{1}{16}e^{\frac{\pi^{2}}{m}}\right)^{2}\alpha(m)^{2}
\]

Again, split into four intervals:
\begin{enumerate}
\item $[\tilde{m}_{\max},\infty)$ : $F_{z_{0},n}\equiv0$.
\item $[2,\tilde{m}_{\max})$ : $\alpha(m)\le4$ and $\frac{1}{16}e^{\frac{\pi^{2}}{m}}<10$
so:
\[
\tilde{F}(m)\le2^{-1}\log\left[400\tilde{\rho_{n}}\right]\Bigg(200\tilde{\rho_{n}}\Bigg)^{2}
\]
\[
\int_{2}^{\tilde{m}_{\max}}F_{z_{0},n}(m)dm\le(\tilde{m}_{\max}-2)c_{5}\tilde{\rho_{n}}^{2}\log\tilde{\rho_{n}}=c_{5}(\tilde{\rho_{n}}\log\tilde{\rho_{n}})^{2}
\]
\item $[a_{n},2)$ : $\alpha(m)\le16m^{-2}$, with some algebra one can get
\[
\tilde{F}(m)\le\tilde{\rho_{n}}^{2}\log\left(\tilde{\rho_{n}}\right)\cdot m^{-2}\cdot e^{c_{6}m^{-1}}=\tilde{\rho_{n}}^{2}\log\left(\tilde{\rho_{n}}\right)\cdot\frac{\mathrm{d}}{\mathrm{d}m}\left[-c_{6}^{-1}e^{c_{6}m^{-1}}\right]
\]
\[
\int_{a_{n}}^{2}F_{z_{0},n}(m)dm\le\int_{a_{n}}^{2}\tilde{F}(m)dm\le c_{6}^{-1}\tilde{\rho_{n}}^{2}\log\left(\tilde{\rho_{n}}\right)\left[-e^{c_{6}m^{-1}}\right]_{a_{n}}^{2}\le
\]
\[
\le c_{6}^{-1}\tilde{\rho_{n}}^{2}\log\left(\tilde{\rho_{n}}\right)\cdot e^{c_{6}a_{n}^{-1}}
\]
we can choose then $a_{n}=\frac{3c_{6}}{\log n}$ and get:
\[
\int_{a_{n}}^{2}F_{z_{0},n}(m)dm\le c_{6}^{-1}\tilde{\rho_{n}}^{2}\log\tilde{\rho_{n}}\cdot\sqrt[3]{n}
\]
\item $(0,a_{n})$ : Here (and everywhere) $F_{z_{0},n}\le n$ so $\int_{0}^{a_{n}}F_{z_{0},n}(m)dm\le3c_{6}\frac{n}{\log n}$.
\end{enumerate}
For $n>1$, $\frac{n}{\log n}>\sqrt[3]{n}>1$ one can find a constant
$C$ such that the sum of all these parts is $\le C\tilde{\rho_{n}}^{2}\log^{2}\tilde{\rho_{n}}\frac{n}{\log n}$
so it ends the proof of Theorem \ref{thm:BoundedType}. If $\inf_{n}\delta_{n}=\delta>0$
then $\frac{1}{n}\int_{0}^{\infty}F_{z_{0},n}(m)dm\le C\left(\frac{S_{f}+M_{f}}{\delta}\right)^{3}\frac{1}{\log n}\to0$
so it is the end of Theorem \ref{thm:MainResult} too.

As for Remark \ref{rem:DeltaGrowth} we can again take $a_{n}$ that
decreases even slower, say $a_{n}=\frac{c_{6}}{\log\log n}$ and get
\[
\frac{1}{n}\int_{a_{n}}^{2}F_{z_{0},n}(m)dm\le c_{6}^{-1}\tilde{\rho_{n}}^{2}\log\tilde{\rho_{n}}\cdot\frac{\log n}{n}
\]

which tends to zero as long as $\tilde{\rho_{n}}=4\frac{M_{f}+S_{f}}{\delta_{n}}\le\kappa n^{\beta}$
for some $\beta<\frac{1}{2}$ and $\kappa>0$.

\subsection*{Acknowledgements}
I am grateful to my M.Sc. supervisor, Prof. Levin Genadi, who assisted
me in each step to complete the research and writing of my master's
thesis \cite{wein2017lyp} and the publication of its results in this
paper. Thanks for hours of guidance and for introducing me to this
wonderful research field. I would also like to thank the referee who
made the publication of this paper possible by constructive, insightful
recommendations.


\begin{thebibliography}{1}

%% Use the widest label as parameter above.
%% Reference items can be numbered or have labels of your choice, as below.
%% Arrange the items in the alphabetical order of names (and not in the order of labels).

%% In IMPAN journals, only the title is italicized; boldface is not used.
%% Do NOT give the issue number unless the issues are paginated separately, as in Uspekhi below.

%%%%%%%%%%% To ease editing, use normal size:

\normalsize
\baselineskip=17pt

%%%%%%%%%%%%%

\bibitem[1]{ahlfors2010conformal} L. V. Ahlfors,
\emph{Conformal invariants: topics in geometric function theory},
 vol. 371, American Mathematical Soc., 2010.

\bibitem[2]{eremenko1992dynamical} A. Eremenko and M. Y. Lyubich,
\emph{Dynamical properties of some classes of entire functions},
Annales de l'institut Fourier, vol. 42, 1992, pp. 989-1020.

\bibitem[3]{levin2014lyapunov} G. Levin, F. Przytycki and W. Shen,
\emph{The Lyapunov exponent of holomorphic maps},
 InventionesMathematicae 205 (2015), no. 2, 363-382.

\bibitem[4]{wein2017lyp}  I. O. Weinstein,
\emph{Pointwise Lyapunov exponent of holomorphic maps},
Master's thesis, Einstein Institute of Mathematics, Faculty of Science, The Hebrew University of Jerusalem, August 2016, pp. 25-40.

\end{thebibliography}
\end{document}